\documentclass[12pt]{amsart}
\usepackage{amscd,amsmath,amsthm,amssymb}
\usepackage{pstcol,pst-plot,pst-3d}
\usepackage{lmodern,pst-node}
\def\NZQ{\Bbb}               

\def\ZZ{{\NZQ Z}}

\def\frk{\frak}               

\def\mm{{\frk m}}

\def\MI{{\mathcal I}}

\def\MF{{\mathcal F}}

\def\MF{{\mathcal F}}
\def\MG{{\mathcal G}}

\def\opn#1#2{\def#1{\operatorname{#2}}} 
\opn\chara{char} \opn\length{\ell} \opn\pd{pd} \opn\rk{rk}
\opn\projdim{proj\,dim} \opn\injdim{inj\,dim} \opn\rank{rank}
\opn\depth{depth} \opn\grade{grade} \opn\height{height}
\opn\embdim{emb\,dim} \opn\codim{codim}

\opn\Tr{Tr} \opn\bigrank{big\,rank}
\opn\superheight{superheight}\opn\lcm{lcm}
\opn\trdeg{tr\,deg}
\opn\reg{reg} \opn\lreg{lreg} \opn\ini{in} \opn\lpd{lpd}
\opn\size{size}\opn\bigsize{bigsize}
\opn\cosize{cosize}\opn\bigcosize{bigcosize}
\opn\sdepth{sdepth}\opn\sreg{sreg}
\opn\link{link}\opn\fdepth{fdepth}
\opn\Deg{Deg}
\opn\div{div} \opn\Div{Div} \opn\cl{cl} \opn\Cl{Cl}
\opn\Spec{Spec} \opn\Supp{Supp} \opn\supp{supp} \opn\Sing{Sing}
\opn\Ass{Ass} \opn\Min{Min}\opn\Mon{Mon} \opn\dstab{dstab} \opn\astab{astab}
\opn\Ann{Ann} \opn\Rad{Rad} \opn\Soc{Soc}
\opn\Im{Im} \opn\Ker{Ker} \opn\Coker{Coker} \opn\Am{Am}
\opn\Hom{Hom} \opn\Tor{Tor} \opn\Ext{Ext} \opn\End{End}
\opn\Aut{Aut} \opn\id{id}

\opn\nat{nat}
\opn\pff{pf}
\opn\Pf{Pf} \opn\GL{GL} \opn\SL{SL} \opn\mod{mod} \opn\ord{ord}
\opn\Gin{Gin} \opn\Hilb{Hilb}\opn\sort{sort}
\opn\aff{aff} \opn\con{conv} \opn\relint{relint} \opn\st{st}
\opn\lk{lk} \opn\cn{cn} \opn\core{core} \opn\vol{vol}
\opn\link{link} \opn\star{star}\opn\lex{lex} \opn\sat{sat}
\opn\gr{gr}

\def\pot#1#2{#1[\kern-0.28ex[#2]\kern-0.28ex]}

\opn\dirlim{\underrightarrow{\lim}}
\opn\inivlim{\underleftarrow{\lim}}
\let\union=\cup

{}
\def\Implies{\ifmmode\Longrightarrow \else
        \unskip${}\Longrightarrow{}$\ignorespaces\fi}
\def\implies{\ifmmode\Rightarrow \else
        \unskip${}\Rightarrow{}$\ignorespaces\fi}
\def\iff{\ifmmode\Longleftrightarrow \else
        \unskip${}\Longleftrightarrow{}$\ignorespaces\fi}

\let\:=\colon
\newtheorem{Theorem}{Theorem}[section]
\newtheorem{Lemma}[Theorem]{Lemma}
\newtheorem{Corollary}[Theorem]{Corollary}
\newtheorem{Proposition}[Theorem]{Proposition}
\newtheorem{Remark}[Theorem]{Remark}

\newtheorem{Example}[Theorem]{Example}

\let\epsilon\varepsilon
\let\kappa=\varkappa
\def\qed{\ifhmode\textqed\fi
      \ifmmode\ifinner\quad\qedsymbol\else\dispqed\fi\fi}
\def\textqed{\unskip\nobreak\penalty50
       \hskip2em\hbox{}\nobreak\hfil\qedsymbol
       \parfillskip=0pt \finalhyphendemerits=0}
\def\dispqed{\rlap{\qquad\qedsymbol}}

\opn\dis{dis}
\def\pnt{{\raise0.5mm\hbox{\large\bf.}}}

\opn\Lex{Lex}

\begin{document}
\title {Linear flags and Koszul filtrations}
\author {Viviana Ene, J\"urgen Herzog and Takayuki Hibi}

\thanks{The  first author was supported by the grant UEFISCDI,  PN-II-ID-PCE-2011-3-1023.}

\address{Faculty of Mathematics and Computer Science, Ovidius University, Bd.\ Mamaia 124,
 900527 Constanta, and
Simion Stoilow Institute of Mathematics of the Romanian Academy,
 Research group of the project  ID-PCE-2011-1023,
 P.O.Box 1-764,  Bucharest 014700,
  Romania} \email{vivian@univ-ovidius.ro}

\address{J\"urgen Herzog, Fachbereich Mathematik, Universit\"at Duisburg-Essen, Campus Essen, 45117
Essen, Germany} \email{juergen.herzog@uni-essen.de}

\address{Takayuki Hibi, Department of Pure and Applied Mathematics, Graduate School of Information Science and Technology,
Osaka University, Toyonaka, Osaka 560-0043, Japan}
\email{hibi@math.sci.osaka-u.ac.jp}

\begin{abstract}
We show that the graded maximal ideal of a graded $K$-algebra $R$ has linear quotients for a suitable choice and  order of its generators if the defining ideal  of $R$ has a quadratic Gr\"obner basis with respect to the reverse lexicographic order, and show that this linear quotient property for algebras defined by binomial edge ideals characterizes closed graphs. Furthermore, for algebras defined by binomial edge ideals attached to a closed graph and for join-meet rings
attached to a finite distributive lattice we present explicit Koszul filtrations.
\end{abstract}
\maketitle

\section*{Introduction}
Let $K$ be a field. In this paper we consider standard graded $K$-algebras. Any such algebra is isomorphic to a $K$-algebra of the form $S/I$ where $S=K[x_1,\ldots,x_n]$ is the polynomial ring in the indeterminates $x_1,\ldots,x_n$ and $I\subset S$ is a graded ideal with $I\subset (x_1,\ldots,x_n)^2$. Let   $\mm$ is the graded maximal ideal of $S/I$. The $K$-algebra $R=S/I$ is called Koszul, if $K=R/\mm$ has a linear resolution. In other words, $R$ is Koszul, if the chain maps in the minimal graded free $R$-resolution of the residue class field of $R$ are given by matrices whose entries are linear forms.

Obviously, the polynomial ring $S$ itself is a Koszul algebra since the Koszul complex attached to the sequence $x_1,\ldots,x_n$ provides a linear (and finite) resolution of $K$. For $R=S/I$ with  $I\neq 0$ the graded minimal free $R$-resolution $F$ of the residue class field $R/\mm$ is infinite and there are examples by Roos \cite{R} which show that $F$ may be linear up to any given homological degree and then becomes non-linear. Thus it is not surprising that no algorithm  for testing Koszulness is known, and in fact, there may not exist any such algorithm. The more it is of interest to have some necessary conditions and also some sufficient conditions for Koszulness at hand. It is well-known that $I$ must be quadratically generated if $S/I$ is Koszul, and that $S/I$ is indeed Koszul if $I$ has a quadratic Gr\"obner basis.

More recently, filtrations have been considered to check whether a standard graded $K$-algebra is Koszul. This strategy has first been applied in the paper \cite{HHR} in which the authors introduced  strongly Koszul algebras which are defined via sequential conditions. Inspired by this concept, Conca, Trung and Valla \cite{CTV} introduced the  more flexible notion of  Koszul filtrations which is defined as follows: let $R$ be a standard graded  $K$-algebra with graded maximal ideal $\mm$. A {\em Koszul filtration} of $R$ is a family $\MF$ of ideals generated by linear forms with the property that $\mm \in \MF$ and that for each $I\in \MF$ with $I\neq 0$ there exists $J\in \MF$ with  $J\subset I$  such that $I/J$ is a cyclic module whose annihilator belongs to $\MF$. It is  shown in \cite[Proposition 1.2]{CTV} that  all the ideals $I$ belonging to a Koszul filtration have a linear  resolution. In particular, any standard graded $K$-algebra admitting a Koszul filtration is Koszul. It has been shown by  an example in \cite[Page 101]{CRV} that not each Koszul algebra has a Koszul filtration.

The question is how the property of a standard graded $K$-algebra to admit a Koszul filtration is related to the property that its defining ideal admits a quadratic Gr\"obner basis? In this paper we will be mainly concerned with this question. At  present it seems to us that none of these properties  implies the other one. Indeed, in Section~\ref{two} we give an example (Example~\ref{example1}) of a binomial edge ideal whose residue class ring has a Koszul filtration,  while  in the given coordinates the ideal has no quadratic Gr\"obner basis for any monomial order. Other examples arise from the work of Ohsugi and Hibi \cite{OH}.  On the other hand, if the Koszul filtration $\MF$ is of a very special nature, namely if $\MF$ is a flag, then the defining ideal of the algebra has a quadratic Gr\"obner basis, see \cite[Theorem 2.4]{CRV}.

For the moment we do not know of any example of a standard graded $K$-algebra which does not admit a Koszul filtration, even though its defining ideal has a quadratic Gr\"obner basis. As a generalization of a lemma of Sturmfels (\cite[Chapter 12]{St}) we prove in Section~\ref{one} the following result (Theorem~\ref{linearquot}): let $I\subset S$ be a  graded  ideal which has a quadratic Gr\"obner basis with respect to the reverse lexicographic order induced by $x_1>\cdots >x_n.$ Then, for all $i,$ the colon ideals $(I,x_{i+1},\ldots,x_n):x_i$ are generated, modulo $I,$ by linear forms. Thus the flag of ideals $0 \subset (\bar{x}_n) \subset (\bar{x}_n,\bar{x}_{n-1})\subset \cdots \subset (\bar{x}_n,\bar{x}_{n-1},\ldots,\bar{x}_1)$ has the potential to belong to a Koszul filtration of $S/I$. Here $\bar{f}$ denotes the residue class of a polynomial $f\in S$ modulo $I$. We call any chain of ideals $(0)=I_0\subset I_1\subset I_2\subset \cdots
\subset I_n=(\bar{x}_1,\ldots,\bar{x}_n)$ generated by linear forms a {\em linear flag} if for all $j$, $I_{j+1}/I_j$ is cyclic and the annihilator of $I_{j+1}/I_j$ is generated by linear forms. Thus Theorem~\ref{linearquot} says that if $I$ has a quadratic Gr\"obner basis with respect to the reverse lexicographic order, then $S/I$ admits a linear flag

In general, even if $I$ is a binomial ideal with quadratic Gr\"obner basis with respect to the reverse lexicographic order,  the colon ideals $(\bar{x}_{i+1},\ldots,\bar{x}_n):\bar{x}_i$ are not generated by subsets of $\{\bar{x}_1,\bar{x}_2,\ldots,\bar{x}_n\}$. However this is the case in various combinatorial contexts, and in particular for toric ideals. This is the content of Theorem~\ref{takayuki} in which we give an algebraic condition for  binomial ideals which ensures that all the colon ideals under consideration are generated by  variables.  There we also show that under the given conditions the colon ideals modulo $I$ and modulo $\ini_<(I)$ are generated by the residue classes of the same sets of variables. This observation makes it much easier to compute these colon ideals and in some cases allows a combinatorial interpretation.

Suppose now that $0 \subset (\bar{x}_n) \subset (\bar{x}_n,\bar{x}_{n-1})\subset \cdots \subset (\bar{x}_n,\bar{x}_{n-1},\ldots,\bar{x}_1)$ is a linear flag. One may ask whether under this assumption $I$ has a  quadratic Gr\"obner basis with respect to the reverse lexicographic order induced by $x_1>\cdots >x_n.$ In general this is not the case. Indeed, let $R_{5,2}=K[x_1,\ldots,x_{10}]/I$ be the $K$-algebra generated by all squarefree monomials $t_it_j\subset K[t_1,\ldots,t_5]$ in 5 variables with $\bar{x}_k=t_{i_k}t_{j_k}$ and such that $k<\ell$ if $t_{i_k}t_{j_k}>t_{i_\ell}t_{j_\ell}$ in the lexicographic order. Then $I$ does not have a quadratic Gr\"obner basis with respect to the reverse lexicographic order induced by $x_1>x_2> \ldots >x_n$. Nevertheless, the sequence $x_n,x_{n-1},\ldots,x_1$ has linear quotients modulo $I$ and hence defines a linear flag. Actually, it is shown in \cite{HQS} that $R_{m,2}$, the algebra generated by all squarefree monomials of degree $2$ in $m$ variables, even has a Koszul filtration.

On the other hand,  in Theorem~\ref{closed} we show that if  $G$ is a finite simple graph on the vertex set $[n]$, and  $J_G\subset K[x_1,\ldots,x_n,y_1,\ldots,y_n]$ is its binomial edge ideal,  then $G$ is closed, i.e. $J_G$ has a quadratic Gr\"obner basis with respect to the  reverse  lexicographic order induced by  $y_1> y_2> \cdots > y_n>x_1> x_2>\cdots > x_n$, if and only if all the colon ideals $(\bar{x}_{i+1},\ldots,\bar{x}_n):\bar{x}_i$ have linear quotients.  In Section~\ref{two} we then show in Theorem~\ref{filtration} that, for a closed graph, the linear flag
$0\subset (\bar{x}_n)\subset \cdots \subset (\bar{x}_n,\bar{x}_{n-1},\ldots,\bar{x}_1)$ can be extended to a Koszul filtration of $S/J_G$. We close this section by proving in Corollary~\ref{hibikoszul} that the family of poset ideals of a finite distributive lattice defines a Koszul filtration of the join-meet ring attached to the lattice. As a consequence one obtains that all the poset ideals generate ideals with  linear resolution in the join-meet ring.

\section{Gröbner bases and linear quotients}
\label{one}

Let $K$ be a field and $S=K[x_1\ldots, x_n]$ the polynomial ring in the variables $x_1,\ldots,x_n$. Any standard graded $K$-algebra $R$ of embedding dimension $n$
is  isomorphic to $S/I$ where $I$ is a graded ideal with $I\subset (x_1,\ldots,x_n)^2$. Let $\mm$ be the graded maximal ideal of $R$. As explained in the introduction, a {\em Koszul filtration} of $R$ is a finite set $\MF$ of ideals  generated by linear forms  such that
\begin{enumerate}
\item[(i)] $\mm\in \MF$;
\item[(ii)] for any $I\in \MF$ with $I\neq 0$, there exists  $J\in \MF$ with $J\subset I$ such that $I/J$ is cyclic and $J:I\in \MF$.
\end{enumerate}

As shown in \cite[Proposition 1.2]{CTV}, any $I\in  \MF$ has a linear resolution and, in particular, $R$ is Koszul if it admits a Koszul filtration.

\medskip
Obviously, if $\MF$ is a Koszul filtration, then $\MF$ contains a flag of ideals
\[
0=I_0\subset I_1\subset I_2\subset \cdots \subset  I_n=\mm,
\]
where $I_j\in \MF$ for all $j$ (and $I_j/I_{j+1}$ is cyclic for all $j$). If it happens that for all $j$ there exists $k$ such that $I_{j+1}:I_j=I_k$, then $\{I_0,I_1,\ldots,I_n\}$ is a Koszul filtration. Such Koszul filtrations are called {\em Koszul flags}. Conca, Rossi and Valla showed  \cite[Theorem 2.4]{CRV} that if $S/I$ has a Koszul flag, then $I$ has a quadratic Gr\"obner basis. The following theorem  is a partial converse of this result.

\begin{Theorem} \label{linearquot}
Let $I\subset S$ be a  graded ideal which has a quadratic Gr\"obner basis with respect to the reverse lexicographic order induced by $x_1>\cdots >x_n.$ Then, for all $i,$ the colon ideals
\[
(I,x_{i+1},\ldots,x_n):x_i
\] are generated, modulo $I,$ by linear forms.
\end{Theorem}

For the proof of the theorem we need to recall the following result from \cite[Chapter 12]{St}.

\begin{Lemma}\label{Sturmfels}
Let $\MG$ be the reduced Gr\"obner basis of the graded ideal $I\subset S$ with respect to the reverse lexicographic order
induced by $x_1>\cdots >x_n.$ Then
\[
\MG^\prime=\{f\in \MG: x_n\not | f\}\cup \{f/x_n: f\in \MG \text { and }x_n | f\}
\]
is a Gr\"obner basis of $I:x_n.$
\end{Lemma}

\begin{proof}[Proof of Theorem~\ref{linearquot}] Let $\MG=\{g_1,\ldots,g_m\}$ be the reduced Gr\"obner  basis of $I$ with respect to the reverse lexicographic order and fix $i\leq n.$ Let $f_j=g_j \mod (x_{i+1},\ldots,x_n),$ $f_j\in K[x_1,\ldots,x_i]$ for all $j.$ We may assume that $\ini_<(g_1)>\cdots> \ini_<(g_m)$ and,
 therefore, that there exists an $s\leq m$ such that $f_s\neq 0$ and $f_{s+1}=\cdots =f_m=0.$ In addition, we have
$\ini_<(f_j)=\ini_<(g_j)$ for $1\leq j\leq s.$ It then follows that  $(I,x_{i+1},\ldots,x_n)=(f_1,\ldots,f_s,x_{i+1},\ldots,x_n)$ and
the set $\MF=\{f_1,\ldots,f_s,x_{i+1},\ldots,x_n\}$ is a Gr\"obner basis since by \cite[Lemma 4.3.7]{HH} $$\ini_<(I,x_{i+1},\ldots,x_n)=(\ini_<(I),x_{i+1},\ldots,x_n).$$ Moreover, $\MF$ is reduced since $\MG$ is reduced. Let $J=(f_1,\ldots,f_s).$  Then
\[
(I,x_{i+1},\ldots,x_n):x_i=(J,x_{i+1},\ldots,x_n):x_i=(J:x_i)+(x_{i+1},\ldots,x_n).
\]
By applying Lemma~\ref{Sturmfels} for $J\cap K[x_1,\ldots, x_i],$ it follows that, modulo $J,$ $(J:x_i)$ is generated by linear forms in $K[x_1,\ldots,x_i]$ which implies that $(I,x_{i+1},\ldots,x_n)$ is also generated by linear forms modulo $I.$
\end{proof}

Consider the ideal $I$ which is  generated by the binomial  $x_1x_3-x_2x_3$.  Then $I:x_3=(I,x_1-x_2)=(x_1-x_2)$. Thus, in general, one can not expect that under the assumptions of Theorem~\ref{linearquot} the ideals $(I,x_{i+1},\ldots,x_n):x_i$ are  generated  by a subset of the variables modulo $I$, even when $I$ is a binomial ideal. Therefore some additional assumptions on the Gr\"obner basis of $I$ are  required to have monomial colon ideals.

For a graded ideal $J\subset S=K[x_1,\ldots,x_n]$ we denote by $J_j$ the $j$th graded component of $J$.

\begin{Theorem}
\label{takayuki}
Let $I\subset S=K[x_1,\ldots,x_n]$ be an ideal generated by quadratic binomials,  and let $<$ be the reverse lexicographic order induced by $x_1>x_2>\cdots >x_n$. Let $f_1,\ldots, f_m$ be the degree $2$ binomials of the reduced Gr\"obner basis of $I$ with respect to $<$. Let  $f_i=u_i-v_i$ for $i=1,\ldots,m$, and assume that $\gcd(u_i,v_i)=1$ for all $i$. Then, for all $i,$ we have:
\begin{enumerate}
\item[(a)] $[(I,x_{i+1}, \ldots,x_n):x_i]_1=[(\ini_<(I),x_{i+1}, \ldots,x_n):x_i]_1$;
\item[(b)] Suppose $I$ has a quadratic Gr\"öbner basis with respect to $<$.
Then
\[
(I,x_{i+1}, \ldots,x_n):x_i=(I,x_{i+1}, \ldots,x_n, \{x_j\:\; j\leq i,\, x_jx_i\in \ini_<(I)\}),
\]
and
\[
(\ini_<(I),x_{i+1}, \ldots,x_n):x_i= (\ini_<(I),x_{i+1}, \ldots,x_n, \{x_j\:\; j\leq i,\, x_jx_i\in \ini_<(I)\}).
\]
\end{enumerate}

\end{Theorem}

\begin{proof}
(a) Let $\ell=\sum_{k=1}^na_kx_k$ be a linear form with $\ell x_i\in (I,x_{i+1}, \ldots,x_n)$. We may assume that $a_k=0$ for $k>i$.
Let $x_j=\ini_<(\ell)$. Then $j\leq i$ and $x_jx_i\in \ini_<(I,x_{i+1}, \ldots,x_n)=(\ini_<(I),x_{i+1}, \ldots,x_n)$. Therefore, there exists $f_k$ with $\ini_<(f_k)=x_jx_i$. Thus, if $f_k=x_jx_i-x_rx_s$, then $s\geq i$. However, since $\gcd(u_k,v_k)=1,$ we see that $s>i$. This implies that $x_jx_i\in (I,x_{i+1},\ldots,x_n)$ and, consequently, $(\ell-a_jx_j)x_i\in  (I,x_{i+1}, \ldots,x_n)$. Since $x_j\in (\ini_<(I),x_{i+1}, \ldots,x_n):x_i$, induction on $\ini_<(\ell)$ shows that $\ell\in (\ini_<(I),x_{i+1}, \ldots,x_n):x_i$.

Conversely, suppose $\ell\in (\ini_<(I),x_{i+1}, \ldots,x_n):x_i$. Since $(\ini_<(I),x_{i+1}, \ldots,x_n)$ is a monomial ideal, we may assume that $\ell$ is a monomial and $\ell\not\in  (\ini_<(I),x_{i+1}, \ldots,x_n)$, say, $\ell=x_j$. Then $j\leq i$ and $x_jx_i\in  (\ini_<(I),x_{i+1}, \ldots,x_n)$. Then, as before, there exists $f_k=x_jx_i-x_rx_s$ with $r\leq s$ and $s>i$. It follows that $x_jx_i\in (I,x_{i+1},\ldots,x_n)$. and hence $x_j\in (I,x_{i+1}, \ldots,x_n):x_i$.

(b) Suppose that ${\mathcal G} = \{f_{1}, \ldots, f_{m}\}$
is the reduced Gr\"obner basis of $I$
with respect to $<$.
Let
$J_{i} = (I,x_{i+1}, \ldots,x_n):x_i$
and
\[
J'_{i} = (I,x_{i+1}, \ldots, x_n, \{x_j\:\; j\leq i,\, x_jx_i\in \ini_<(I)\}).
\]
One has $J'_{i} \subset J_{i}$.  To see why this is true, suppose that
$x_jx_i\in \ini_<(I)$ with $j \leq i$.
Then there is $f_{k} = x_{j}x_{i} - x_{p}x_{q} \in {\mathcal G}$ with
$\ini_{<}(f) = x_{j}x_{i}$.
Since $j \leq i$, it follows that either $p > i$ or $q > i$.
Hence $x_{p}x_{q} \in (I,x_{i+1}, \ldots,x_n)$.
Thus $x_{j}x_{i} \in (I,x_{i+1}, \ldots,x_n)$
and $x_{j} \in J_{i}$, as required.

Now, let ${\mathcal A}$ denote the set of
homogeneous polynomials $f \in S$ of degree $\geq 1$ which belongs to
$J_{i}$ with the property
that none of the monomials appearing in $f$ belongs to $J'_{i}$.
Suppose that ${\mathcal A} \neq \emptyset$.
Among the polynomials belonging to ${\mathcal A}$,
we choose $f \in {\mathcal A}$ such that
$\ini_{<}(f) \leq \ini_{<}(g)$ for all $g \in {\mathcal A}$.
Let $u = \ini_{<}(f)$.  Since $x_{i}f \in (I,x_{i+1}, \ldots,x_n)$,
one has $x_{i} u \in (\ini_{<}(I),x_{i+1}, \ldots,x_n)$.
Since $u \not\in J'_{i}$, it follows that $x_{i} u \in \ini_{<}(I)$.
Thus there is $f_{\ell} = x_{p}x_{q} - x_{r}x_{s}$ with $\ini_{<}(f) = x_{p}x_{q}$
such that $x_{p}x_{q}$ divides $x_{i} u$.
If, say, $p = i$, then $x_{q}$ divides $u$.
Thus $q \leq i$.  Since $x_{i}x_{q} \in \ini_{<}(I)$,
one has $x_{q} \in J'_{i}$.  This contradict $u \not\in J'_{i}$.
Thus $p \neq i$, $q \neq i$ and
$x_{p}x_{q}$ divides $u$.  Let $w = (u/x_{p}x_{q})x_{r}x_{s}$
and $f' = f - a(u - w)$, where $a \neq 0$ is the coefficient of $u$ in $f$.
Since $u - w \in I$, one has $f' \in J_{i}$.
Since $u \not\in J'_{i}$, one has $w \not\in J'_{i}$.
Thus $f' \in {\mathcal A}$ and
$\ini_{<}(f') < \ini_{<}(f)$.
This contradicts the choice of $f \in {\mathcal A}$.
Hence ${\mathcal A} = \emptyset$ and $J_{i} = J'_{i}$, as desired.

The proof of the corresponding statement for $\ini_<(I)$ is obvious.
\end{proof}

In \cite{C} a standard graded $K$-algebra $R$ is a called {\em universally Koszul}, if the set  consisting of all ideals generated  by linear forms is a Koszul filtration of $R$. In combinatorial contexts it is natural to consider  standard graded $K$-algebras $R$ whose   set of ideals consisting of all ideals which are generated by subsets of the variables is a Koszul filtration of $R$. We call such algebras {\em c-universally Koszul}.  It is clear that any
universally Koszul algebra or any strongly Koszul algebra is c-universally Koszul.

\begin{Corollary}
\label{toric}
Let $I\subset S=K[x_1,\ldots,x_n]$ be a toric ideal with the property that  $I$ has a quadratic Gr\"obner basis with respect to the reverse lexicographic order  induced by any given order of the variables. Then  $S/I$ is c-universally Koszul.
\end{Corollary}

\begin{proof} The binomials in a minimal set of binomial generators of a torc ideal are all irreducible, since $I$ is a prime ideal. Hence, the conclusion follows immediately from Theorem~\ref{takayuki}.
\end{proof}

\begin{Remark}\label{not} {\em In view of Theorem~\ref{linearquot} and Theorem~\ref{takayuki} one may expect that the following more general statement may be true: let $I\subset S=K[x_1,\ldots,x_n]$ be a graded  ideal. Then the following are equivalent:  (a) $I$ has a quadratic Gr\"obner basis with respect to the reverse lexicographic order induced by $x_1>x_2>\cdots >x_n$, (b) $x_n,x_{n-1},\ldots,x_1$ has linear quotients modulo $I$. In general however, (b) does not imply (a). Indeed, let $R_{5,2}=K[x_1,\ldots,x_{10}]/I$ the $K$-algebra generated by all squarefree monomials $t_it_j\subset K[t_1,\ldots,t_5]$ in $5$ variables with $\bar{x}_k=t_{i_k}t_{j_k}$ and such that $k<\ell$ if $t_{i_k}t_{j_k}>t_{i_\ell}t_{j_\ell}$ in the lexicographic order. Then $I$ does not have a quadratic Gr\"obner basis with respect to the reverse lexicographic order induced by $x_1>x_2> \ldots >x_n$. Nevertheless, the sequence $x_n,x_{n-1},\ldots,x_1$ has linear quotients modulo $I$. }
\end{Remark}

\medskip
Surprisingly,  for any binomial edge ideal, the conditions (a) and (b) as formulated in Remark~\ref{not} turn out to be equivalent, as will be shown in the next theorem.

Let $G$ be a finite simple graph on the vertex set $[n]$.  The {\em binomial edge ideal} $J_G$ associated with $G$ is the ideal generated by the quadrics $f_{ij}= x_iy_j-x_jy_i$ in $S=K[x_1,\ldots,x_n, y_1,\ldots,y_n]$  with $\{i,j\}$ an edge of $G$. This class of ideals was introduced in \cite{HHHKR} and \cite{Oht}.

The graph  $G$  is called {\em closed} with respect to the given labeling if $G$  satisfies the following condition: whenever $\{i, j\}$ and $\{i, k\}$
are edges of $G$  and either $i < j$,  $i < k$  or $i > j$,  $i > k$ then $\{j,k\}$  is also an edge of $G$. It is shown in
\cite[Theorem 1.1]{HHHKR} that $G$ is closed with respect to the given labeling if and only if $J_G$ has a quadratic Gr\"
obner basis with respect to the lexicographic order induced by $x_1>\cdots>x_n>y_1>\cdots >y_n$. It is easily seen that any
binomial ideal $J_G$ has the same reduced Gr\"obner basis with respect to the lexicographic order induced by the natural
order of the variables and with respect to the reverse lexicographic order induced by $y_1>\cdots >y_n>x_1>\cdots >x_n.$
Therefore, $G$ is closed with respect to the given labeling if and only if $J_G$ has a quadratic Gr\"obner basis with
respect to the reverse lexicographic order induced by $y_1>\cdots>y_n>x_1>\cdots >x_n$.

One calls a graph $G$ closed if it is closed with respect to some labeling of its vertices.  D.A. Cox and A. Erskine \cite[Theorem 1.4]{CE} showed that a connected graph $G$ is closed if and only if $G$ is chordal, claw-free and narrow.

\medskip
We will often use the following notation. For $k\in [n]$, we let
\[
N^{<}(k)=\{j: j<k,\ \{j,k\}\in E(G)\} \text{ and } N^{>}(k)=\{j: j>k,\ \{k,j\} \in E(G)\}.
\]

For some of the following proofs it will be useful to note that, provided that they are nonempty,  each of these sets are intervals if the graph $G$ is closed with respect to its labeling. Indeed, let us take $i\in N^{<}(k)$. In particular, we have $\{i,k\}\in E(G).$ Then,
as all the maximal cliques of $G$ are intervals (see \cite[Theorem 2.2]{EHH1}), it follows that for any $i\leq j< k,$ $\{j,k\}\in E(G)$, thus
$j\in N^{<}(k).$ A similar argument works for $N^{>}(k).$

\begin{Theorem}
\label{closed}
Let $G$ be a connected finite simple graph on the vertex set $[n]$. The following conditions are equivalent:
\begin{enumerate}
\item[(a)] $G$ is closed with respect to the given labeling;
\item[(b)]  the sequence $x_n,x_{n-1},\ldots, x_1$ has linear quotients modulo $J_G$.
\end{enumerate}
\end{Theorem}

\begin{proof}
(a)\implies (b): 
Let $G$ be closed with respect to the given labeling. It follows that the generators of $J_G$ form the reduced Gr\"obner basis of $J_G$
 with respect to the reverse lexicographic order induced by $y_1>\cdots >y_n>x_1>\cdots >x_n.$ Let $i\leq n.$ The generators of $\ini_<(J_G)$ which are divisible by $x_i$ are exactly $x_iy_j$ where $i<j$ and $\{i,j\}\in E(G).$
Hence, by using Theorem~\ref{takayuki} (b), we get
\begin{eqnarray}
\label{quotients}
(\bar{x}_n,\bar{x}_{n-1},\ldots,\bar{x}_{i+1}):\bar{x}_i=(\bar{x}_n,\bar{x}_{n-1},\ldots,\bar{x}_{i+1}, \{ \bar{y}_{j}\:\; j\in N^{>}(i)\}).
\end{eqnarray}
Here $\bar{f}$ denotes the residue class for a polynomial $f\in S$ modulo $J_G$.

(b)\implies (a): We may suppose that $\bar{x}_n,\bar{x}_{n-1},\ldots,\bar{x}_1$ has linear quotients and show that $G$ is closed with respect to the given labeling. In fact, assume that $G$ is not closed. Then there exist $\{i,j\}, \{i,k\}\in E(G)$ with $i<j<k$ or $i>j>k$ and such that $\{j,k\}\not\in E(G)$.

Let us first consider the case that $i<j<k$. Since
\[
\bar{x}_j\bar{y}_i\bar{y}_k=\bar{x}_i\bar{y}_j\bar{y}_k=\bar{x}_k\bar{y}_i\bar{y}_j,
\]
 we see that  $\bar{y}_i\bar{y}_k\in (\bar{x}_n,\ldots,\bar{x}_{j+1}):\bar{x}_j$.

 We claim that $\bar{y}_i\bar{y}_k$ is a minimal generator of $(\bar{x}_n,\ldots,\bar{x}_{j+1}):\bar{x}_j$, contradicting the assumption that  $\bar{x}_n,\bar{x}_{n-1},\ldots,\bar{x}_1$ has linear quotients. Indeed, suppose that $\bar{y}_i\bar{y}_k$ is not   a minimal generator of $(\bar{x}_n,\ldots,\bar{x}_{j+1}):\bar{x}_j$, then there exist linear forms $\ell_1$ and $\ell_2$ in $S$  such that $\bar{\ell}_1\bar{\ell}_2=\bar{y}_i\bar{y}_k$ and at least one of the forms $\bar{\ell}_1,\bar{\ell}_2$ belongs to   $(\bar{x}_n,\ldots,\bar{x}_{j+1}):\bar{x}_j$.

 Now we observe that $J_G$ is $\ZZ^n$-graded with $\deg x_i=\deg y_i=\epsilon_i$ for all $i$, where $\epsilon_i$ is the $i$th canonical unit vector of $\ZZ^n$.  It follows that the $\bar{\ell}_i$ are  multi-homogeneous as well with $\deg \bar{\ell}_1\bar{\ell}_2 =\epsilon_i+\epsilon_k$, say  $\deg \bar{\ell}_1=\epsilon_i$ and $\deg \bar{\ell}_2=\epsilon_k$. Thus $\ell_1=ax_i+by_i$ and $\ell_2=cx_k+dy_k$ with $a,b,c,d\in K$. Let us first assume that $\bar{\ell}_1\in (\bar{x}_n,\ldots,\bar{x}_{j+1}):\bar{x}_j.$ We get
\[
ax_ix_j+bx_jy_i\in (J_G,x_n,\ldots,x_{j+1})
\]
which implies that
\[
\ini_<(ax_ix_j+bx_jy_i)\in \ini_<(J_G,x_n,\ldots,x_{j+1})=((\ini_<J_G),x_n,\ldots,x_{j+1}).
\]
Here $<$ denotes the reverse lexicographic order induced by $y_1>\cdots >y_n>x_1>\cdots >x_n.$  It follows that
$x_ix_j\in \ini_<(J_G)$ or $x_jy_i\in\ini_<(J_G)$, which is impossible since the generators of degree $2$ of $\ini_<(J_G)$ are
of the form $x_ky_\ell$ with $\{k,\ell\}\in E(G)$ and $k<\ell.$

Let us consider now that $\bar{\ell}_2\in (\bar{x}_n,\ldots,\bar{x}_{j+1}):\bar{x}_j.$ We get $cx_kx_j+d x_jy_k\in (J_G,x_n,\ldots,x_{j+1}).$ If $d\neq 0,$ we obtain $x_jy_k\in (J_G,x_n,\ldots,x_{j+1})$ and, therefore, $x_jy_k\in (\ini_<(J_G),x_n,\ldots,x_{j+1})$ which implies that $x_jy_k\in \ini_<(J_G)$, a contradiction since $\{j,k\}\not\in E(G)$ by assumption. Therefore, we must have $\ell_2=cx_k$ for some $c\in K\setminus\{0\}.$ The equation $\bar{\ell}_1\bar{\ell}_2=\bar{y}_i\bar{y}_k$ implies that $cx_k(ax_i+by_i)-y_iy_k\in J_G.$ This implies that one of the monomials $x_ix_k, x_ky_i, y_iy_k$ belongs to
$\ini_<(J_G)$, contradiction.

Finally we consider the case that $i>j>k$. Then $x_if_{jk}\in J_G$, and so $\bar{f}_{jk}\in (\bar{x}_n,\ldots,\bar{x}_{i+1}):\bar{x}_i$. By similar arguments as above we show that $\bar{f}_{jk}$ is a minimal generator of $(\bar{x}_n,\ldots,\bar{x}_{i+1}):\bar{x}_i$. Suppose that there exist linear forms $\ell_1=ax_j+by_j$ and $\ell_2=cx_k+dy_k$ such that $g=f_{jk}-\ell_1\ell_2\in J_G$. Since no monomial in the support of $g$ belongs to $\ini_<(G)$ (with the monomial order as in the previous paragraph), it follows that $g\not\in J_G$, a contradiction. Hence we see that  $(\bar{x}_n,\ldots,\bar{x}_{i+1}):\bar{x}_i$ is not generated by linear forms.
\end{proof}

\section{Classes of ideals with Koszul filtration}
\label{two}

In this section we present two large classes of $K$-algebras which admit Koszul filtrations. In both cases their defining ideal also admits a quadratic Gr\"obner basis.

\begin{Theorem}\label{filtration}
Let $G$ be a closed graph. Then $R=S/J_G$ has a Koszul filtration.
\end{Theorem}

For the proof of this theorem we need a preparatory result.

\begin{Lemma}\label{casetwo}
Let $0\leq k\leq n-1$, $N^>(k)=\{k+1,\ldots,\ell\}$, for some $\ell\geq k+1,$ and $N^<(k+1)=\{i,i+1,\ldots,k\}$
 for some $i\leq k.$ Then:
\begin{itemize}
\item [(a)]
$
(J_G,x_n,\ldots,x_{k+1},y_{k+2},\ldots,y_{\ell}):y_{k+1}=(J_G,x_n,\ldots,x_{k+1},x_k,\dots, x_i,y_{k+2},\ldots,y_\ell),
$
\item [(b)] For $k+2\leq s\leq \ell,$ $y_s$ is regular on
$
(J_G,x_n,\ldots,x_i,y_{s+1},\ldots,y_\ell).
$
\end{itemize}
\end{Lemma}

\begin{proof}
(a). Let $r\in N^<(k+1).$ Then $x_ry_{k+1}=(x_ry_{k+1}-x_{k+1}y_r)+x_{k+1}y_r\in (J_G,x_n,\ldots,x_{k+1}).$ This shows the inclusion $\supseteq.$

For the other inclusion, let $f\in S$ such that
$
f y_{k+1}\in(J_G,x_n,\ldots,x_{k+1},y_{k+2},\ldots,y_\ell).
$
 If $H$ is the restriction of $G$ to the set $[k],$ then
$
(J_G,x_n,\ldots,x_{k+1},y_{k+2},\ldots,y_\ell)=(J_H,x_n,\ldots,x_{k+1},y_{k+2},\ldots,y_\ell,\{x_ry_j: r\leq k<j, \{r,j\}\in E(G)\}).
$
Let us observe that, if $\{r,j\}\in E(G)$ with $r\leq k<j,$ then, as $G$ is closed, we have $\{k,j\}\in E(G),$ thus $j\in
\{k+1,\ldots,\ell\}.$ Therefore, we get
\[
(J_G,x_n,\ldots,x_{k+1},y_{k+2},\ldots,y_\ell)=(J_H,x_n,\ldots,x_{k+1},y_{k+2},\ldots,y_\ell,x_iy_{k+1},\ldots,x_ky_{k+1}).
\]
By inspecting the $S$--polynomials of the generators in the right side of the above equality of ideals, it follows that
\[
\ini_<(J_G,x_n,\ldots,x_{k+1},y_{k+2},\ldots,y_\ell)=
\]
\[(\ini_<(J_H),x_n,\ldots,x_{k+1},y_{k+2},\ldots,y_\ell,x_iy_{k+1},\ldots,x_ky_{k+1}).
\] Here $<$ denotes the lexicographic order on $S=K[x_1,\ldots,x_n,y_1,\ldots,y_n]$ induced by the natural order of the variables.

It follows that $\ini_<(f)y_{k+1}\in (\ini_<(J_H),x_n,\ldots,x_{k+1},y_{k+2},\ldots,y_\ell,x_iy_{k+1},\ldots,x_ky_{k+1})$
 which implies that $\ini_<(f)\in (\ini_<(J_H),x_n,\ldots,x_{k+1},x_k,\ldots,x_i,y_{k+2},\ldots,y_\ell).$ Hence, either
$\ini_<(f)\in (x_n,\ldots,x_{k+1},x_k,\ldots,x_i,y_{k+2},\ldots,y_\ell)$ or $\ini_<(f)\in \ini_<(J_H).$ In both cases we may
proceed by induction on $\ini_<(f).$ In the first case, let $a$ be the coefficient of  $\ini_<(f)$ in $f$. Then
$g=f-a\ini_<(f)$ has $\ini_<(g)<\ini_<(f)$ and $gy_{k+1}\in (J_G,x_n,\ldots,x_{k+1},y_{k+2},\ldots,y_\ell).$ In the second case,
let $h\in J_H$ and  $c\in K\setminus\{0\}$ such that $\ini_<(h-cf)<\ini_<(f).$ Thus, if $g=h-cf, $ it follows that  $g y_{k+1}\in
(J_G,x_n,\ldots,x_{k+1},y_{k+2},\ldots,y_\ell)$ as well.

(b). Let $k+2\leq s\leq \ell.$ It is enough to show that $y_s$ is regular on the initial ideal of $(J_G,x_n,\ldots,x_i,y_{s+1},\ldots,y_\ell).$ If $H$ is the restriction of $G$ to the set $[i],$ then we get:
\[
\ini_<(J_G,x_n,\ldots,x_i,y_{s+1},\ldots,y_\ell)=
\]
\[\ini_<(J_H,x_n,\ldots,x_i,y_{s+1},\ldots,y_\ell, \{x_ry_j: r<i<j,\{r,j\}\in E(G)\})=
\]
\[
(\ini_<(J_H),x_n,\ldots,x_i,y_{s+1},\ldots,y_\ell, \{x_ry_j: r<i<j,\{r,j\}\in E(G)\}).
\]
The last equality from above may be easily checked by observing that the $S$--polynomials $S(f_{r\ell},x_ry_j)$ reduce to $0$
 for
any $r<\ell\leq i<j$ with $\{r,\ell\}\in E(H).$ We claim that $y_s$ does not divide any of the generators of
$$(\ini_<(J_H),x_n,\ldots,x_i,y_{s+1},\ldots,y_\ell, \{x_ry_j: r<i<j,\{r,j\}\in E(G)\}).$$ Obviously, $y_s$ does not
divide any of the generators of $\ini_<(J_H).$ Next, if $\{r,s\}\in E(G)$ for some $r<i<k+1<s$, then, as $G$ is closed, we get
$\{r,k+1\}\in E(G)$, contradiction to the fact that $i=\min N^<(k+1).$ This shows that none of the generators $x_ry_j$ is divisible by $y_s$.
\end{proof}

\begin{proof}[Proof of Theorem~\ref{filtration}]
Let $G$ be closed with respect to its labeling. We set $\bar{f}$ for $f \mod (J_G)\in R=S/J_G.$
For $k\in [n-1]$, let $N^{>}(k)=\{k+1,\ldots,\ell_k\}$ and $N^{<}(k+1)=\{i_k,i_k+1,\ldots,k\}$.

Let us consider the following families of ideals:
\[
\MF_1=\{(\bar{x}_n,\ldots,\bar{x}_1,\bar{y}_n,\ldots,\bar{y}_{k}): 1\leq k\leq n\}\cup\{(\bar{x}_n,\ldots,\bar{x}_k):1\leq k\leq n\},
\]
\[
\MF_2=\bigcup_{k=1}^{n-1}\{(\bar{x}_n,\ldots,\bar{x}_{k+1},\bar{y}_{k+1},\ldots,\bar{y}_{\ell_k}),
(\bar{x}_n,\ldots,\bar{x}_{k+1},\bar{y}_{k+2},\ldots,\bar{y}_{\ell_k})\},
\]
and
\[
\MF_3=\bigcup_{k=1}^{n-1}\{(\bar{x}_n,\ldots,\bar{x}_{i_k},\bar{y}_{s},\ldots,\bar{y}_{\ell_k}): k+2\leq s\leq \ell_k\}.
\]

We claim that the family
$\MF=\MF_1\cup \MF_2\cup \MF_3\cup\{(0)\}$  is a Koszul filtration of $R$. We have to check that, for every $I\in \MF,$
there exists $J\in \MF$ such that $I/J$ is cyclic and $J:I\in \MF.$

Let us  consider $I=(\bar{x}_n,\ldots,\bar{x}_1,\bar{y}_n,\ldots,\bar{y}_{k})\in \MF_1.$ Then, for
$J=(\bar{x}_n,\ldots,\bar{x}_1,$ $ \bar{y}_n,\ldots,  \bar{y}_{k+1})\in \MF_1,$ we have  $J:I=J$ since $\bar{y}_k$ is obviously
regular on $R/J.$

For $I=(\bar{x}_n,\ldots,\bar{x}_k)\in \MF_1$ with $1\leq k\leq n-1,$ we take $J=(\bar{x}_n,\ldots,\bar{x}_{k+1})\in \MF_1$.
Then,  by (\ref{quotients}), we get $J:I=(\bar{x}_n,\ldots,\bar{x}_{k+1},\bar{y}_{k+1},\ldots,\bar{y}_{\ell_k})\in \MF_2.$
In addition, for $I=(\bar{x}_n),$ we  have $(0):I=(0)$ since $\bar{x}_n$ is regular on $R.$

Let us now choose $I\in \MF_2,$ $I=(\bar{x}_n,\ldots,\bar{x}_{k+1},\bar{y}_{k+1},\ldots,\bar{y}_{\ell_k})$ for some
$1\leq k\leq n-1.$ Then, $J=(\bar{x}_n,\ldots,\bar{x}_{k+1},\bar{y}_{k+2},\ldots,\bar{y}_{\ell_k})\in \MF_2$ and, by Lemma~\ref{casetwo} (a), we have $J:I=(\bar{x}_n,\ldots,\bar{x}_{i_k},\bar{y}_{k+2},\ldots,\bar{y}_{\ell_k})\in \MF_3.$

Finally, if $I\in \MF_3,$ $I=(\bar{x}_n,\ldots,\bar{x}_{i_k},\bar{y}_{s},\ldots,\bar{y}_{\ell_k})$ for some $k+2\leq s\leq \ell_k,$ we take $J=(\bar{x}_n,\ldots,\bar{x}_{i_k},\bar{y}_{s+1},\ldots,\bar{y}_{\ell_k})\in \MF_3.$ By Lemma~\ref{casetwo} (b), we get $J:I=J$ since $\bar{y}_s$ is regular on $R/J.$
\end{proof}

One may ask for which binomial edge ideals $J_G$ the $K$-algebra  $S/J_G$ is c-universally Koszul?   The following result answers this question.

\begin{Proposition}
\label{complete}
Let $G$ be a finite simple graph. Then $S/J_G$ is c-universally Koszul, if and only if $G$ is a complete graph.
\end{Proposition}

\begin{proof} In \cite[Example 1.6]{HHR} it has been shown that $S/J_G$ is strongly Koszul if $G$ is a complete graph.  In particular, $G$ is c-universally Koszul. For the converse implication, assume that $G$ is not complete. Then there exist edges $\{i,j\}$ and $\{i,k\}$ of $G$ such that $\{j,k\}\not\in E(G)$. This implies that $x_jy_k-x_ky_j\not \in J_G$. On the other hand, $x_jy_k-x_ky_j \in J_G:x_i$, because $x_i(x_jy_k-x_ky_j)=x_j(x_iy_k-x_ky_i)-x_k(x_iy_j-x_jy_i)$.
\end{proof}

The following example shows that the converse of Theorem~\ref{filtration} is not true. In other words,   there exist Koszul non-closed
graphs $G$ such that $R=S/J_G$ has a Koszul filtration.

\begin{Example}\label{example1}{\em
Let $G$ be the graph given in Figure~\ref{example}. $G$ is not closed, but it is Koszul \cite{EHH}.

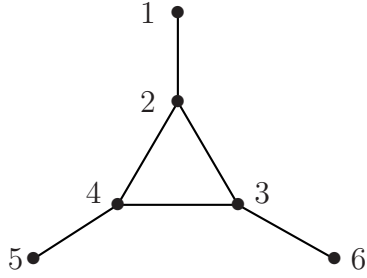
\begin{figure}[hbt]
\begin{center}
\psset{unit=0.8cm}
\begin{pspicture}(1,-2)(5,5)
\pspolygon(2,2)(3,3.71)(4,2)
\psline(3,3.71)(3,5.2)
\psline(0.6,1.1)(2,2)
\psline(4,2)(5.6,1.1)

\rput(2,2){$\bullet$} \rput(1.6,2.2){$4$}
\rput(3,3.71){$\bullet$} \rput(2.5,3.71){$2$}
\rput(4,2){$\bullet$} \rput(4.4,2.2){$3$}

\rput(3,5.2){$\bullet$} \rput(2.5,5.2){$1$}
\rput(0.6,1.1){$\bullet$} \rput(0.3,1.1){$5$}
\rput(5.6,1.1){$\bullet$} \rput(6,1.1){$6$}

\end{pspicture}
\end{center}
\label{example}
\end{figure}

The ring $R=K[x_1,\ldots,x_6,y_1,\ldots,y_6]/J_G$ possesses the following Koszul filtration:

{\footnotesize
\[
\begin{array}{lll}
	(0), &(y_6), & (y_6,x_6),\\
	(y_6,y_3), & (y_6,x_6,x_5), & (y_6,x_6,y_5,x_5),\\
	(y_6,x_6,x_5,x_4), & (y_6,y_4,x_6,x_5,x_4), & (y_6,x_6,x_5,x_4,x_3),\\
	(y_6,x_6,x_5,x_4,x_3,x_2), & (y_6,y_4,x_6,x_5,x_4,x_3),& (y_6,y_4,y_3,x_6,x_5,x_4,x_3),\\
	(y_6,x_6,x_5,\ldots,x_1), & (y_6,y_2,x_6,x_5,\ldots,x_2), & (y_6,y_4,x_6,x_5,\ldots,x_2),\\
	(y_6,y_5,x_6,x_5,\ldots,x_1), & (y_6,y_5,y_4,x_6,x_5,\ldots,x_1), & (y_6,y_5,y_4,y_3x_6,x_5,\ldots,x_1),\\
	(y_6,y_5,\ldots,y_2,x_6,x_5,\ldots,x_1),& (y_6,y_5,\ldots,y_1,x_6,x_5,\ldots,x_1).
\end{array}
\]}
}
\end{Example}

In view of this example it would be of interest to classify all finite simple graphs for which $S/J_G$ has a Koszul filtration.

\medskip

Let $K$ be field and $L$ be a finite distributive lattice. We denote $K[L]$ the polynomial ring over $K$ whose variables are the elements of $L$ and denote by $I_L$ the binomial ideal in $K[L]$ generated by the binomials $ab-(a\wedge b)(a\vee b)$ with $a,b\in L$ incomparable. The ideal $I_L$ is called the {\em join-meet ideal} and  $A(L)=K[L]/I_L$ the {\em join-meet ring} of $L$ (also known as  Hibi ring).  The residue class $f+I_L\in A(L)$ of a polynomial $f\in K[L]$ will be denoted $\bar{f}$. Let $I\subset L$ be a poset ideal of $L$.   We denote by $\bar{I}$ the ideal in $A(L)$ generated by the elements $\bar{a}$ with $a\in I$.

\begin{Theorem}
\label{hibicolon}
Let $I\subset J$ be poset ideals of $L$ with $J\setminus I=\{a\}$. Then $\bar{I}:\bar{J}=\bar{H}$, where $H$ is the poset ideal $\{b\in L\: b \ngtr a\}$.
\end{Theorem}

\begin{proof}
Let  $b\in L$ with $b \ngtr a$. Then $\bar{a}\bar{b}=(\bar{a}\wedge \bar{b})(\bar{a}\vee \bar{b})\in \bar{I}$, since $a\wedge b<a$. This shows that $\bar{H}\subset \bar{I}:(\bar{a})=\bar{I}:\bar{J}$. In order  to prove  the converse inclusion, we let $\bar{f}\in \bar{I}:(\bar{a})$ and may assume that $\bar{f}\not\in \bar{I}$. Thus $af \in (I_L, I)$ and $f\notin (I_L, I)$. Now we choose a total order $\prec$ of the variables such that $a\prec b$ if $a>b$ in $L$ and such that $a\prec b$ if $a\notin I$ and $b\in I$, and denote again by $\prec$ the reverse lexicographic order induced by $\prec$. Then $\ini_\prec(I_L,I)=(\ini_\prec(I_L),I)$,  and it follows that $a\ini_\prec(f)\in (\ini_\prec(I_L),I)$. By a classical result of Hibi (see \cite[Theorem 10.1.3]{HH}, $\ini_\prec(I_L)$ is generated by all monomials $bc$ with $b,c\in L$ incomparable. Thus
\[(\ini_\prec(I_L),I)=(\{bc\:\; b,c\in L\setminus I \text{ with $b,c$ incomparable}\}, I).
\]
Since $f\notin (I_L, I)$, we may assume that $f$ is in standard form with respect to $(I_L, I)$ and $\prec$.
In other words, we may assume that no monomial in the support of $f$ belongs to $(\ini_\prec(I_L),I)$. On the other hand, since $a\ini(f)\in (\ini_\prec(I_L),I)$ it follows that one of the generating monomials of $(\ini_\prec(I_L),I)$ divides $a\ini_<(f)$. The only monomials among the monomial generators which can divide $a\ini_\prec(f)$ must be of the form $ab$ with $a$ and $b$ incomparable. Thus $b$ divides $\ini_\prec(f)$ and $b\in H$. Let $g=f-\lambda \ini_\prec(f)$ where $\lambda$ is the leading coefficient of $f$. Since  $\bar{g}\in \bar{I}:(\bar{a})$ and $\ini_\prec(g)\prec \ini_\prec(f)$, induction completes the proof
\end{proof}

\begin{Corollary}
\label{hibikoszul}
Let $L$ be a finite distributive lattice. Then the  family $$\MF=\{\bar{I}\:\; \text{$I$ is a poset ideal of $L$}\}$$ of ideals is a Koszul filtration of $A(L)$. In particular, for each poset ideal $I$ of $L$,  the ideal $\bar{I}\subset A(L)$ has a linear resolution.
\end{Corollary}

In \cite{HHR} all finite distributive lattices $L$  for which $A(L)$ is strongly Koszul are classified. Among them are the Boolean lattices. Thus  if $B$ is a Boolean lattice, then $B$ admits  the Koszul filtration consisting of all  ideals of the form $\bar{U}$ with $U$ a subset of $B$, and by Corollary~\ref{hibikoszul}, $B$  also admits the Koszul filtration  consisting of all poset ideals of $B$. These are already two different Koszul filtrations of $A(B)$.

An {\em upset} in a partially ordered set $P$ is a subset $J$ with the property that if $x\in J$ and $y\geq x$, then $y\in J$. Since reversion of the partial order in a distributive lattice $L$ defines again  a distributive lattice, it follows from Corollary~\ref{hibikoszul}, that the collection of ideals $\bar{J}$ with $J\subset L$ an upset form a Koszul filtration. So for any Boolean lattice we have now three different Koszul filtrations. One obtains even more Koszul filtrations by observing that if $\MF_1$ and $\MF_2$ are Koszul filtrations of  a standard graded $K$-algebra $R$, then $\MF_1\union \MF_2$ is a Koszul filtration of $R$ as well. Thus any standard graded $K$-algebra $R$ which admits a Koszul filtration, also admits a Koszul filtration  which among all Koszul filtrations  of $R$ is {\em maximal}  with respect to inclusion. The maximal Koszul filtration is of interest because it gives a large family of ideals of linear forms with linear resolution.

We say a Koszul filtration $\MF$ of $R$ is {\em minimal}, if no proper subset of $\MF$ is a Koszul filtration of $R$.  In general the Koszul filtration of $A(L)$ given in Corollary~\ref{hibikoszul} is not minimal. To see this, first notice that each of the poset ideals   $H$ in  Theorem~\ref{hibicolon} is co-generated by a single element. Thus the following observation is immediate.  Suppose $\MI$ is a set of poset ideals of $L$ satisfying the following conditions:
\begin{enumerate}
\item[(i)] all poset ideals co-generated by an element  of $L$,  and  $L$ itself belong to $\MI$;
\item[(ii)] for  all $I\in\MI$ there exists $J\subset I$ such that $|I\setminus J|=1$.
\end{enumerate}
Then $\MF=\{\bar{I}\:\; I\in\MI\}$ is a Koszul filtration of $A(L)$.

In general a set $\MI$ of poset ideals of $L$ satisfying (i) and (ii) may be different from the set of all poset ideals of $L$. For example, let $B_3$ be the Boolean lattice of rank $3$ whose elements we may identify with the subsets of $[3]$.
Then the set $\MF$ consisting of all poset ideals of $B_3$
except the poset ideal $\{\{3\},\emptyset\}$ satisfies (i) and (ii).

\end{document}